\documentclass[11pt]{article}

\usepackage{graphicx}
\usepackage[reqno]{amsmath}
\usepackage{amssymb}
\usepackage{amsthm}
\usepackage{color}
\usepackage{dutchcal}

\usepackage{enumitem}
\usepackage{cleveref}

\usepackage{relsize}
\usepackage{nth}    % abbreviation for numbering, e.g. first, second,..

\numberwithin{equation}{section}

\newtheorem{theorem}{Theorem}[section]
\newtheorem{lemma}[theorem]{Lemma}

\newtheorem{proposition}[theorem]{Proposition}

\newtheorem{remark}[theorem]{Remark}
\newtheorem{example}[theorem]{Example}

\newcommand{\R}{\mathbb{R}}

\newcommand{\B}{\mathbb{B}}

\newcommand{\xb}{\bar{x}}

\makeatletter
\newcommand{\customlabel}[2]{%
	\protected@write \@auxout {}{\string \newlabel {#1}{{#2}{}}}}
\makeatother

\makeatletter
\newcommand{\leqnomode}{\tagsleft@true\let\veqno\@@leqno}
\newcommand{\reqnomode}{\tagsleft@false\let\veqno\@@eqno}
\makeatother

\date{}
\begin{document}

\title{Normality of Necessary Optimality Conditions for Calculus of Variations Problems with State Constraints}

	\author{ 
		N. Khalil\footnote{ {\it MODAL'X, Universit\'e Paris Ouest Nanterre La D\'efense, 200 Avenue de la R\'epublique, 92001 Paris Nanterre, France, e-mail: \/}
			{\tt nathalie.khalil@parisnanterre.fr}} ,
		S. O. Lopes \footnote{ {\it CFIS and DMA, Universidade do Minho, Guimar$\tilde{a}$es, Portugal, e-mail: \/} 
			{\tt sofialopes@math.uminho.pt}
				This author was supported by POCI-01-0145-FEDER-006933-SYSTEC, PTDC/EEI-AUT/2933/2014, POCI-01-0145-FEDER-016858 TOCCATTA and POCI-01-0145-FEDER-028247  To Chair - funded by FEDER funds through COMPETE2020 - Programa Operacional Competitividade e Internacionaliza\c{c}$\tilde{a}$o (POCI) and by national funds (PIDDAC) through FCT/MCTES which is gratefully acknowledged. Financial support from the Portuguese Foundation for Science and Technology (FCT) in the framework of the Strategic Financing UID/FIS/04650/2013 is also acknowledged. \protect \includegraphics[height=5.0mm]{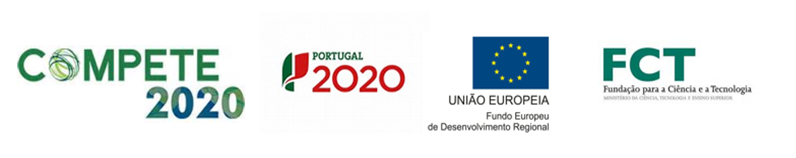}  }
	}

	\maketitle
\begin{abstract}\noindent
We consider non-autonomous calculus of variations problems with a state constraint represented by a given closed set. We prove that if the interior of the Clarke tangent cone of the state constraint set is non-empty (this is the constraint qualification that we suggest here), then the necessary optimality conditions apply in the normal form. We establish normality results for (weak) local minimizers and global minimizers, employing two different approaches and invoking slightly diverse assumptions. More precisely, for the local minimizers result, the Lagrangian is supposed to be Lipschitz with respect to the state variable, and just lower semicontinuous in its third variable. On the other hand, the approach for the global minimizers result (which is simpler) requires the Lagrangian to be convex with respect to its third variable, but the Lipschitz constant of the Lagrangian with respect to the state variable might now depend on time.

%We consider non-autonomous calculus of variations problems with a state constraint represented by a given closed set. We prove that if the interior of the Clarke tangent cone of the state constraint set is non-empty (this is the constraint qualification that we suggest here), then the necessary optimality conditions apply in the normal form. We establish normality results for $W^{1,1}-$local minimizers and global minimizers, employing two different approaches and invoking slightly diverse assumptions. More precisely, for the $W^{1,1}-$local minimizers result, the Lagrangian $L= L(t,x,v)$ is supposed to be Lipschitz w.r.t. the state variable $x$, and just lower semicontinuous w.r.t. $v$. On the other hand, the approach for the global minimizers result (which is simpler) requires the Lagrangian to be convex w.r.t. $v$, but the Lipschitz constant of $L$ w.r.t. $x$ might now depend on time.  

%We give the details of two possible techniques for the proof: the first one concerns $W^{1,1}-$local minimizers, and uses the normality results for optimal control problems once our variational problem is expressed as an optimal control problem with final cost (via a {\it state augmentation} procedure); the second technique which involves $W^{1,1}-$global minimizers, is simpler and employs a neighboring feasible trajectory result with $L^\infty-$linear estimates.
\end{abstract}


\section{Introduction}
We consider the following non-autonomous calculus of variations problem subject to a state constraint 
\begin{equation}\begin{cases}\label{problem: calculus of variations}
		\begin{aligned}
			& {\text{minimize}}
			&&  \int_S^{T} L (t,x(t),\dot{x}(t)) \ dt  \\
			&&& \hspace{-1.9cm} \text{over arcs } x(.) \in  W^{1,1}([S,T], \mathbb{R}^{n}) \text{ satisfying} \\
			&&& x(S) = x_0 \ , \\
			&&& x(t) \in  A \quad \quad {\rm for \, all }\; t \in [S,T] \  , \end{aligned}\end{cases} \tag{CV}\leqnomode
\end{equation}
for a given Lagrangian $L: [S,T]\times \R^n \times \R^n \to \R$, an initial datum $x_0 \in \R^n$ and a closed set $A \subset \mathbb{R}^{n}$.

Briefly stated, the problem consists in minimizing the integral of a time-dependent Lagrangian $L$ over admissible absolutely continuous arcs $x(.)$ (that is, the left end-point and the state constraint of the problem (\ref{problem: calculus of variations}) are satisfied). We say that an admissible arc $\bar x$ is a $W^{1,1}-${\it local minimizer} if there exists $\epsilon > 0$ such that
\[
\int_S^{T} L (t,\bar{x} (t),\dot{\bar x}(t))dt \leq \int_S^{T} L (t,x(t),\dot{x}(t))dt ,
\] for all admissible arcs $x$ satisfying
\[
\left\| x(.) - \bar x (.) \right\|_{W^{1,1}} \leq \epsilon .
\]
%where $\|x(.)\|_{W^{1,1}} := |x(S)| + \| \dot x(.)\|_{L^{1}} $ for all $x(.) \in W^{1,1}([S,T], \mathbb{R}^n).$

The purpose of this paper is to derive necessary optimality conditions in the normal form for problems like (\ref{problem: calculus of variations}); that is, when trajectories satisfy a set of necessary conditions with a nonzero cost multiplier. Indeed, in the pathological situation of abnormal extrema, the objective function to minimize does not intervene in the selection of the candidates to be minimizers. To overcome this difficulty, new additional hypotheses have to be imposed, known as {\it constraint qualifications}, that permit to identify some class of problems for which normality is guaranteed. There has been a growing interest in the literature to ensure the normality of necessary optimality conditions for state-constrained calculus of variations problems. For instance, \cite{ferreira_when_1994} deals with autonomous Lagrangian, studied for $W^{1,1}-$local minimizers and with a state constraint expressed in terms of an inequality of a twice continuously differentiable function. The constraint qualification referred to these smooth problems imposes that the gradient of the function representing the state constraint set is not zero at any point on the boundary of the state constraint set. The result in \cite{ferreira_when_1994} has been successively extended in \cite{fontes_normal_2013} to the nonsmooth case, for $L^\infty-$local minimizers, imposing a constraint qualification which makes use of some hybrid subgradients to cover situations in which the function, which defines the state constraint set, is not differentiable. More precisely the idea of the constraint qualification in \cite{fontes_normal_2013} is the following: the angle between any couple of (hybrid) subgradients of the function that defines the state constraint set is `acute'.
A useful technique employed in \cite{ferreira_when_1994} and \cite{fontes_normal_2013} in order to derive optimality conditions in the normal form consists in introducing an extra variable which reduces the reference calculus of variations problem to an optimal control problem with a terminal cost (this method is known as the `state augmentation').

In this paper, the state constraint set is given in the intrinsic form (i.e. it is a given closed set) and the constraint qualification we suggest is to assume that the interior of the Clarke tangent cone to the state constraint set is nonempty.  We emphasize that our constraint qualification generalizes the ones discussed in \cite{fontes_normal_2013} and \cite{ferreira_nondegenerate_1999} and can be therefore applied to a broader class of problems. This is clarified in Proposition \ref{Prop3} and Example \ref{ex0} where the constraint qualifications reported in \cite{fontes_normal_2013} and \cite{ferreira_nondegenerate_1999} imply the one that we suggest in our paper.

We propose two main theorems which establish the normality of the necessary optimality conditions for $W^{1,1}-$local minimizers and for global minimizers. For these two cases we identify two different approaches, both based on a state augmentation technique, but combined with:
\begin{enumerate}[label= \arabic*), ref= \arabic*)]
	\item\label{item: approach 1) for normality in CV} either a construction of a suitable control and a normality result for optimal control problems (this is for the $W^{1,1}-$local minimizers case);
	\item\label{item: approach 2) for normality in CV} or a `distance estimate' result coupled with a standard maximum principle in optimal control (for the global minimizers case). 
\end{enumerate}

More precisely, in the first result approach (considering the case of $W^{1,1}-$local minimizers), we invoke some stability properties of the interior of the Clarke tangent cone. This allows to select a particular control which pushes the dynamic of the control system inside the state constraint more than the reference minimizer. In such circumstances, necessary conditions in optimal control apply in the normal form: we opt here for the broader version (i.e. covering $W^{1,1}-$local minimizers) of normality for optimal control problem established in \cite{fontes_normal_2013} (which is concerned merely with $L^\infty-$local minimizers). The normality of the associated optimal control problem yields the desired normality property for our reference calculus of variations problem, considering $W^{1,1}-$local minimizers. This result is valid for Lagrangians $L=L(t,x,v)$ which are Lipschitz w.r.t. $x$ and lower semicontinuous w.r.t. $v$. The second result approach (which deals with global minimizers) requires to impose slightly different assumptions on the data: here the Lipschitz constant of the Lagrangian w.r.t. $x$ might be an integrable function depending on the time variable, but $v \mapsto L(t,x,v)$ has to be convex. It is also necessary to impose a constraint qualification which is slightly stronger than that one considered for the first result (but still in the same spirit). The proof in this case is much shorter, simpler, and it employs a neighboring feasible trajectories result satisfying $L^\infty-$linear estimates (cf. \cite{rampazzo_theorem_1999} and \cite{bettiol_l$infty$_2012}). This permits to find a global minimizer for an auxiliary optimal control problem to which we apply a standard maximum principle. The required normality form of the necessary conditions for the reference problem in calculus of variations can be therefore derived.

%The study of the regularity of the minimizers for problems expressed in the form of calculus of variations in the autonomous and the non-autonomous case has also attracted the attention of many mathematicians, like Clarke, Dal Maso, Frankowska, Mariconda, Vinter etc. We mention a brief overview of their works (merely in the autonomous case): Clarke and Vinter \cite{clarke_regularity_1985} have established that if the Lagrangian has some regularity properties (mainly Lipschitz) and verifies convexity w.r.t. the dynamic, then the minimizer enjoys a Lipschitz property. This result was extended to the non-convex Lagrangian (in the dynamics) and weaker assumptions on the regularity in a paper by Dal Maso and Frankowska \cite{dal_maso_autonomous_2003}.

%Another class of calculus of variations problems with boundary conditions was also studied in \cite{bousquet_continuity_2011}, \cite{bousquet_local_2007} and \cite{mariconda_lipschitz_2007} where some nice properties regarding the regularity of the minimizer are established. The existence of solutions for calculus of variations problems was also postulated in many papers, for instance in \cite{cellina_existence_1994} (where some growth conditions are assumed, without a convexity assumption w.r.t. the dynamics) and \cite{mariconda_existence_2002} (where some boundary conditions exist and no growth assumptions on $L$ are needed).

\vspace{0.5cm}

This paper is organized as follows. In a short preliminary section, we provide some of the basic notions and results of the nonsmooth analysis that are used throughout the paper. Section 3 provides the main results of this paper: necessary optimality conditions in a normal form for non-autonomous calculus of variations problems, for $W^{1,1}-$local and global minimizers. The last two sections are devoted to prove our main results making use of a normality result for optimal control problems (for the case of $W^{1,1}-$local minimizers) and a neighboring feasible trajectory result (for the case of global minimizers).

\vskip2ex
\noindent
{\bf Notation.} In this paper, $| . |$ refers to the Euclidean norm. $\mathbb{B}$ denotes the closed unit ball. Given a subset $X$ of $\mathbb{R}^{n}$, $\partial X$ is the boundary of $X$, ${\rm int}\ X$ is the interior of $X$ and ${\rm co}\ X$ is the convex hull of $X$. The Euclidean distance of a point $y$ from the set $X$ (that is $\inf \limits_{x \in X} | x-y | $) is written $d_X(y)$. Given a measure $\mu$ on the time interval $[S,T]$, we write $supp(\mu)$ for the support of the measure $\mu$. $W^{1,1}([S,T], \mathbb{R}^n)$ denotes the space of absolutely continuous functions defined on $[S,T]$ and taking values in $\mathbb{R}^{n}$.

\section{Some basic nonsmooth analysis tools}
We introduce here some basic nonsmooth analysis tools which will be employed in this paper. For further details, we refer the reader for instance to the books \cite{aubin_set-valued_2009}, \cite{clarke_functional_2013}, \cite{clarke_nonsmooth_2008}, \cite{clarke_optimization_1990}, \cite{mordukhovich_variational_2006} and \cite{vinter_optimal_2010}.
\vskip2ex
\noindent

Take a closed set $A \subset \mathbb{R}^{n}$ and a point $\bar x \in A$. The \textit{proximal normal cone} of $A$ at $\bar{x}$ is defined as
\begin{eqnarray*}
		N_{A}^{P}(\bar{x}) := \{ \eta \in \mathbb{R}^{n} : \text{there exists } M > 0 \text{ such that } \eta \cdot (x-\bar{x}) \leq M \left\| x - \bar{x} \right\| ^{2} \text{ for all } x \in A \} \ .
\end{eqnarray*}

The \textit{limiting normal cone} of $A$ at $\bar{x} \in A$ is defined to be
\begin{eqnarray*}
	N_A(\bar{x}) := \lbrace \eta \in \mathbb{R}^n : \text{there exists sequences } x_i \xrightarrow {A} \bar{x} \text{ and } \eta _i \rightarrow \eta \text{ such that } \eta_i \in N_{A}^{P}(x_{i}) \text{ for all } i \rbrace.
\end{eqnarray*}
Here $x_{i} \xrightarrow []{A} \bar{x}$ means that $x_{i} \rightarrow \bar{x}$ and $x_{i} \in A, \text{ for all } i.$

Given a lower semicontinuous function $f: \mathbb{R}^n \longrightarrow \mathbb{R} \cup \lbrace \infty \rbrace$, the \textit{limiting subdifferential} of $f$ at a point $\bar{x} \in \mathbb{R}^{n}$ such that $f(\bar{x}) < +\infty$ can be expressed in terms of the limiting normal cone to the epigraph of $f$ as follows:
\[  \partial f(\bar x) := \{ \eta \ : \ (\eta,-1) \in N_{\text{epi }f} (\bar{x}, f(\bar{x}))    \} \ .   \]

%\begin{eqnarray*}
%	&&\partial ^P f(\bar{x}) := \lbrace \eta \in \mathbb{R}^n : \exists \; \epsilon >0 \ , \ M \geq 0 \, \text{ such that }\  \\ &&\hspace{0.5 in} \eta \cdot (x-\bar{x}) \leq f(x)-f(\bar{x}) + M {| x-\bar{x} |}^2,  \text{ for all } x \in \bar{x} + \epsilon \mathbb{B}\rbrace.
%\end{eqnarray*}
%\noindent
%The \textit{limiting subdifferential} of $f$ at $\bar{x} \in \mathbb{R}^{n}$ such that $f(\bar{x}) < +\infty$ is the set
%\begin{eqnarray*}
%	&&\partial f (\bar{x}) := \lbrace \eta \in \mathbb{R}^n : \text{there exists } x_i \xrightarrow[]{f} \bar x \text{ and } \eta_i \to \eta \text{ such that }
%	\eta_i \in \partial^P f(x_i) \text{ for all } i \rbrace,
%\end{eqnarray*}
%where $x_i \xrightarrow[]{f} \bar x$ denotes $x_i  \rightarrow \bar x$ and $f(x_i) \rightarrow f(\bar x)$ for all $i$.

We also make use of the \textit{hybrid subdifferential} of a Lipschitz continuous function $h: \R^n \to \R$ on a neighborhood of $\bar x$, defined as:
\[
\partial^{>}h(\bar x) := {\rm co }\ \lbrace \gamma \in \mathbb{R}^{n} : \text{there exists } \; x_{i} \xrightarrow[]{h} \bar x \text{ such that } h(x_{i}) > 0, \; \text{for all } i \text{ and } \nabla h(x_{i}) \rightarrow \gamma \rbrace,
\]
where $x_i \xrightarrow[]{h} \bar x$ denotes $x_i  \rightarrow \bar x$ and $h(x_i) \rightarrow h(\bar x)$ for all $i$.
\noindent
Moreover, we make reference to the Clarke tangent cone to the closed set $A$ at $\bar x \in A$, defined as follow:
\begin{align*}
T_A(\bar{x}):= \lbrace \xi \in \mathbb{R}^n : \; \text{ for all} \; x_i \xrightarrow[]{A} \bar{x} & \text{ and } \  t_i \downarrow 0,  \text{ there exist }  \lbrace a_i \rbrace  \text{ in } A \text{ such that } \;  t_i^ {-1} (a_i- x_i) \rightarrow \xi \rbrace.
\end{align*}

We present also the following result:
		\begin{proposition} \label{Prop2}
			\begin{itemize}
				\item[(i)] $\partial^{>} d_A(a) \subset {\rm co}\left(N_A(a) \cap \partial \mathbb{B}\right)$ for a closed set $A \subset \mathbb{R}^{n}$ and $a \in \partial A$.
				\item[(ii)] If in addition co $N_A(a)$ is pointed, then $0 \not \in  {\rm co} \left( N_A(a) \cap \partial \mathbb{B}\right)$.
			\end{itemize} \noindent
			(We recall that $ { \rm co} \ N_{A}(a)$ is `pointed' if for any nonzero elements $d_{1}, d_{2} \in  {\rm co}\ N_{A}(a),$
			$
			d_{1} + d_{2} \neq 0.
			$)
		\end{proposition}
\begin{proof}
	The proof of (i) follows directly from the definition of the hybrid subdifferentialof the distance function $d_A(.)$, and from the Caratheodory Theorem; while (ii) is proved by contradiction. 
\end{proof}
\section{Main Results}
This section provides the main results of this paper: necessary optimality conditions in a normal form for non-autonomous calculus of variations problems in the form of (\ref{problem: calculus of variations}). More precisely, we are interested in providing normality conditions for problems in which the state constraint is given in an implicit form: $A$ is just a closed set. Using the distance function, the state constraint $x(t) \in A$ can be equivalently written as a pathwise functional inequality
\[ d_{A}(x(t)) \leq 0 \quad \quad {\rm for \; all }\ t \in [S,T] \ . \]

\noindent
We provide first a normality result for $W^{1,1}-$local minimizers. We assume therefore that for a given reference arc $\bar x \in W^{1,1}([S,T], \R^n)$:
\begin{enumerate}[label=(CV\arabic*), ref=CV\arabic*]
	\item \label{item: CV1} $L(t,x,v)$ is measurable in $(t,v)$, bounded on bounded sets; and there exist $\epsilon'>0$, $K_L>0$ such that for all $t \in [S,T]$ and $x, x' \in \bar{x} (t)+ \epsilon'\B$
	$$
	|L(t,x,v)-L(t,x',v)| \leq K_L |x-x'| \quad \text{uniformly on } v \in \R^n.$$
	\item \label{item: CV2} $v \mapsto L(t,\bar{x}(t),v)$ is lower semicontinuous for all $t \in [S,T]$.
	
%	\item \label{item: CV3} (Coercivity) There exists an increasing function $\theta : [ 0 , \infty) \to [0 , \infty)$ such that $\lim\limits_{\alpha \to \infty} \frac{\theta(\alpha)}{\alpha} = + \infty$, and a constant $\beta$ such that
%	\[  L(t,x,v) > \theta(|v|) - \beta |v|  \quad \text{for all } t \in [S,T], \ x\in \R^n, \ v \in \R^n  . \]
\end{enumerate}

\begin{theorem}[$W^{1,1}-$Local Minimizers]\label{Thm1}
	Let $\bar{x}$ be a $W^{1,1}-$local minimizer for (\ref{problem: calculus of variations}), and assume that hypotheses (\ref{item: CV1})-(\ref{item: CV2}) are satisfied. Suppose also that $\bar x$ is Lipschitz continuous and
	\begin{enumerate}[label= $(CQ)$, ref=$CQ$]
		\item \label{CQ for calculus of variations} \[\text{int } T_A(z) \neq \emptyset \ , \quad \text{for all } z \in \bar{x}([S,T])\cap \partial A.\]
	\end{enumerate}
	Then, there exist $p(.) \in W^{1,1}([S,T], \mathbb{R}^n)$, a function of bounded variation $\nu(\cdot) : [S,T] \rightarrow \mathbb{R}^n$, continuous from the right on $(S,T)$, such that: for some positive Borel measure $\mu$ on $[S,T]$, whose support satisfies
	\[ \text{supp}(\mu) \ \subset \ \{  t \in [S,T] \ : \ \bar x(t) \in \partial A   \},  \]
	and some Borel measurable selection
	\[   \gamma(t) \in \partial_x^> d_A(\bar x(t)) \quad \mu-\text{a.e. }\ t \in [S,T]  \]
	we have
	\begin{enumerate}[label=(\alph*), ref=\alph*]
		\item \label{0_chap5} $\nu(t) = \int_{[S,t]} \gamma(s) d\mu(s)$ \quad for all $t\in (S,T]$ ;
		\item \label{1_chap5} $\dot{p}(t) \in \textrm{co } \partial_{x} L (t,\bar{x}(t), \dot{\bar{x}}(t)) \quad \textrm{ and } \quad q(t) \in \textrm{co } \partial_{\dot{x}} L (t,\bar{x}(t), \dot{\bar{x}}(t))$ \quad a.e. $t\in [S,T]$ ;
		\item \label{2_chap5} $q(T)=0$ ;
	\end{enumerate}	
	where
	\begin{equation*} q(t)  = \begin{cases}
			p(S)   \quad & t =S \\
			p(t) +\int_{[S,t]} \gamma(s) d\mu(s)   \quad &t \in (S,T] \ .
		\end{cases}
	\end{equation*}
\end{theorem}

The following theorem concerns global minimizers, establishing the same necessary optimality conditions (in the normal form) as Theorem \ref{Thm1}. Here we invoke slightly different assumptions: we impose the convexity of the Lagrangian w.r.t. $v$, and the stronger constraint qualification (\ref{CQ for calculus of variations_second_technique}) below (valid for any point belonging to the boundary of $A$); however, we allow that the Lipschitz constant of $L$ w.r.t. $x$ might depend on $t$.
More precisely, we assume that for some $R_0 >0$,

\begin{enumerate}[label=(CV\arabic*)$'$, ref=(CV\arabic*)$'$]
	\item \label{item: CV1'} $L(.,x,v)$ is measurable for all $x\in \R^n$ and $v \in \R^n$, and there exists $\eta > 0$ such that the set-valued map $t \leadsto \{ L(t,x,v) : v \in R_0 \B \}$ is absolutely continuous from the left uniformly over $x \in (\partial A + \eta \B) \cap R_0 \B$. Moreover, $L(t,x,u)$ is bounded on bounded sets and, there exists an integrable function $K_L(.): [S,T] \to \R_+$ such that for all $t \in [S,T]$ and $x, x' \in R_0 \mathbb{B}$
	$$
	|L(t,x,v)-L(t,x',v)| \leq K_L(t) |x-x'| \quad \text{uniformly in } v\in \R^n. $$
	\item \label{item: CV2'} $v \mapsto L(t,x,v)$ is convex, for all $(t,x) \in [S,T] \times \mathbb{R}^n$.
\end{enumerate}

We recall that, given a set $X_0 \subset \R^n$ and a multifunction $F(.,.): [S,T] \times \R^n \leadsto \R^n$, we say that $F(.,x)$ is absolutely continuous from the left, uniformly over $x\in X_0$ if and only if the following condition is satisfied: given any $\epsilon>0$, we may find $\delta>0$ such that, for any finite partition of $[S,T]$
\[ S \le s_1 < t_1 \le s_2 < t_2 \le \ldots \le s_m < t_m \le T   \]  
satisfying $\sum_{i=1}^{m} (t_i-s_i) < \delta$, we have
\[ \sum_{i=1}^{m} d_{F(t_i,x)} (F(s_i,x)) < \epsilon.  \]

%\begin{remark} \label{rmk0} Hypotheses \ref{item: CV1'} and \ref{item: CV2'} imply that $v \mapsto L(t,x,v)$ is locally Lipschitz (cf. \cite[Proposition 2.2.6]{clarke_optimization_1990}).\end{remark}

\begin{theorem}[Global Minimizers]\label{Thm2}
	Let $\bar{x}$ be a global minimizer for (\ref{problem: calculus of variations}), and assume that hypotheses \ref{item: CV1'}-\ref{item: CV2'} are satisfied. Suppose also that $\bar x$ is Lipschitz continuous and
	\begin{enumerate}[label=($\widetilde{CQ}$), ref= $\widetilde{CQ}$]
		\item \label{CQ for calculus of variations_second_technique} \[\text{int } T_A(z) \neq \emptyset \ , \quad \text{for all } z \in \partial A.\]
	\end{enumerate}
	Then, there exist $p(.) \in W^{1,1}([S,T], \mathbb{R}^n)$, a function of bounded variation $\nu(\cdot) : [S,T] \rightarrow \mathbb{R}^n$, continuous from the right on $(S,T)$, such that 
%	for some positive Borel measure $\mu$ on $[S,T]$, whose support satisfies
%	\[ \text{supp}(\mu) \ \subset \ \{  t \in [S,T] \ : \ \bar x(t) \in \partial A   \},  \]
%	and some Borel measurable selection
%	\[   \gamma(t) \in \partial_x^> d_A(\bar x(t)) \quad \mu-\text{a.e. }\ t \in [S,T]  \]
 conditions (\ref{0_chap5})-(\ref{2_chap5}) of Theorem \ref{Thm1} remain valid.
%	\begin{enumerate}[label=(\alph*), ref=\alph*]
%		\item \label{0_second} $\nu(t) = \int_{[S,t)} \gamma(s) d\mu(s)$ \quad for all $t\in (S,T]$ ;
%		\item \label{1_second} $\dot{p}(t) \in \textrm{co } \partial_{x} L (\bar{x}(t), \dot{\bar{x}}(t)) \quad \textrm{ and } \quad q(t) \in \textrm{co } \partial_{\dot{x}} L (\bar{x}(t), \dot{\bar{x}}(t))$ \quad a.e. $t\in [S,T]$ ;
%		\item \label{2_second} $q(T)=0$ ;
%		%\item \label{eq:necessary conditions of optimality10} $\gamma(t) \in  \partial^{>}d_A( \bar{x}(t))$ and $\textrm{supp } \{\mu\} \subset \lbrace t\in [S,T]\; : \; \bar{x}(t) \in \partial A \rbrace\ ;$
%	\end{enumerate}	
%	where
%	\begin{equation*} q(t)  = \begin{cases}
%			p(S)   \quad & t =S \\
%			p(t) +\int_{[S,t]} \gamma(s) d\mu(s)   \quad &t \in (S,T] \ .
%		\end{cases}
%	\end{equation*}
\end{theorem}

\begin{remark}
	We observe that the validity of Theorem \ref{Thm1} and Theorem \ref{Thm2} requires that the minimizer (local or global) $\bar x(.)$ is Lipschitz. This assumption is not restrictive, indeed, not only covers previous results in this framework (as \cite{fontes_normal_2013} and \cite{ferreira_when_1994}), but earlier work shows that Lipschitz regularity of minimizers can be obtained (for both cases of autonomous and non-autonomous Lagrangian) under unrestrictive assumptions on the data. This is not a peculiar property of some problems in the absence of state constraints (see for instance \cite{dal_maso_autonomous_2003}, \cite[Theorem 4.5.2 and Theorem 4.5.4]{clarke_necessary_2005}, \cite[Corollary 3.2]{clarke_regularity_1985}). It is well-known also in the framework of state-constrained calculus of variations problems, see \cite[Corollary 16.19]{clarke_functional_2013}, \cite[Theorem 11.5.1]{vinter_optimal_2010} (for the autonomous case), and more recently in \cite[Theorem 5.2]{bettiol_nonautonomous_2017} (for the nonautonomous case). 
%	In the case of free state-constrained calculus of variations problems, result on the regularity of the minimizers for the non-autonomous Lagrangian can be found in \cite{clarke_regularity_1985}, \cite{clarke_necessary_2005}
\end{remark}

 We recall that the problem of normality of necessary conditions in calculus of variations has been previously investigated in \cite{ferreira_when_1994}, for the case of autonomous Lagrangian and smooth state constraint expressed in terms of a scalar inequality function $\{x \ : \ h(x) \leq 0 \}$ (that is, the function $h$ is of class $C^{2}$). This result has been successively extended to the case where the scalar function $h$ is nonsmooth in \cite{fontes_normal_2013}, always in the framework of autonomous Lagrangian. In both papers \cite{fontes_normal_2013} and \cite{ferreira_when_1994}, the Lagrangian is supposed to be Lipschitz w.r.t. $x$ and convex w.r.t. $v$. In our paper we deal with a non-autonomous Lagrangian and with a constraint qualification of different nature when we consider a state constraint condition in terms of a merely closed set $A$.

%More precisely, the state constraint $A$ is expressed in the form of a functional inequality (namely, $A:= \{ x \ : \ h(x) \leq 0 \}$ where $h$ is a Lipschitz continuous function); and the following constraint qualification is considered for a reference $L^{\infty}-$local minimizer $\bar{x}$ for (\ref{problem: calculus of variations}):
%\begin{enumerate}[label=(CQ13'), ref=CQ13']
%	\item \label{CQ13} There exist positive constants $c, \; \varepsilon$ such that for all $\tau \in \lbrace \sigma \in [S,T] : h(\bar{x}(\sigma)) =0 \rbrace$ and for all $x_{1}, \; x_{2} \in \lbrace \bar{x}(\sigma) : \sigma \in (\tau - \varepsilon , \tau ] \cap [S,T] \rbrace$
%	\[
%	\gamma_{1} \cdot \gamma_{2} > c \ , \quad \quad \text{for all } \ \gamma_{1} \in \partial^{>} h (x_{1}) \ , \quad \gamma_{2} \in \partial^{>} h (x_{2}).
%	\]
%	
%	Furthermore, if $h(\bar{x}(S))=0$, then for all $x_{1}, \; x_{2} \in  \bar{x}(S)  + \varepsilon \mathbb{B}$
%	\[
%	\gamma_{1} \cdot \gamma_{2} > c \ , \quad \quad \text{for all } \ \gamma_{1} \in \partial^{>} h (x_{1}) \ , \quad \gamma_{2} \in \partial^{>} h (x_{2}).
%	\]
%\end{enumerate}
%
%In \cite{fontes_normal_2013}, it was proved that if we assume that hypotheses (\ref{item: CV1})-(\ref{item: CV3}) and (\ref{CQ13}) are satisfied, then the necessary conditions apply in the normal form.

%The following proposition proves that our constraints qualifications (\ref{CQ for calculus of variations}) and (\ref{CQ for calculus of variations_second_technique}) are weaker than (\ref{CQ13}) (we consider $h(.)=d_A(.)$).

In the following proposition and example, it is shown that the constraint qualification invoked in \cite{fontes_normal_2013} (when we consider the state constraint to be expressed in terms of a distance function, that is $h(.) : = d_A(.)$) implies our constraints qualifications (\ref{CQ for calculus of variations}) and (\ref{CQ for calculus of variations_second_technique}), but the reverse implication is not true. Accordingly, (\ref{CQ for calculus of variations}) and (\ref{CQ for calculus of variations_second_technique}) can be applied to a larger class of problems.

%are weaker than the constraint qualification invoked in \cite{fontes_normal_2013} (when we consider the state constraint to be expressed in terms of a distance function $d_A(.)$).

\begin{proposition} \label{Prop3} Consider a closed set $A \subset \mathbb{R}^{n}$ and assume that the distance function to $A$, $d_{A}(.)$, satisfies the following condition: if for any given $y \in \partial A$ there exists $c >0$ such that
	\begin{equation} \label{simpler}
	\gamma_{1} \cdot \gamma_{2} > c \ , \quad \quad \text{for all } \ \gamma_{1} \ , \gamma_{2} \in \partial^{>}d_{A}(y),
	\end{equation}
	then, \;	$ {\textrm int }\ T_{A}(y) \neq \emptyset.$
	
	\noindent(Condition (\ref{simpler}) is a slightly weaker version of the constraint qualification considered in \cite{fontes_normal_2013}.)
\end{proposition}

\begin{proof}[\textbf{Proof of Proposition \ref{Prop3}}]
	Since the hypothesis (\ref{simpler}) considered here clearly implies that $0 \notin \partial ^{>} d_{A}(y)$, then the proof uses the same ideas of \cite[Proposition 2.1]{fontes_normality_2015} which consists in showing that the cone $\R^{+} (\partial ^{>} d_{A}(y))$ is closed and pointed and that it is also equal to ${\rm co }\ N_{A}(y)$. That is the convex hull of the limiting normal cone to $A$ is pointed and, consequently, its polar, $T_{A}(y)$ has a nonempty interior.
\end{proof}

\begin{example} \label{ex0} Consider the set $A := \{ (x,y) \ : \ h(x,y) \leq 0  \}$, where $h:\R^2 \rightarrow \R$ such  that
	\[ h(x,y) = |y| - x. \]
	It is straightforward to check that $\partial^{>} h (0,0)= {\rm co }\ \{ \gamma_1, \gamma_2 \}$, where $ \gamma_1:=(-1,+1)$, and $ \gamma_2:=(-1,-1)$. Therefore, at the point $(0,0)$, condition (\ref{simpler}) is violated when a minimizer $\bar{x}$ is such that $(0,0) \in \bar{x}([S,T])$. This is because we could find two vectors $\gamma_1$ and $\gamma_2$ such that $\gamma_1 \cdot \gamma_2 =0$. However, $\text{int }T_A(0,0) \neq \emptyset$, and more in general $\text{int }T_A(y) \neq \emptyset$ for all $y\in \partial A$. Then, (\ref{CQ for calculus of variations}) (and also (\ref{CQ for calculus of variations_second_technique})) is always satisfied.
\end{example}

This is a simple example which shows that we can find a state constraint set $A$ defined by a functional inequality (for some Lipschitz function $h$) such that (\ref{CQ for calculus of variations}) (and also (\ref{CQ for calculus of variations_second_technique})) is always verified but (\ref{simpler}) fails to hold true when a minimizer goes in a region where $A$ is nonsmooth.

\section{Proof of Theorem \ref{Thm1} ($W^{1,1}-$Local Minimizers)}
Along this section, we identify a class of  optimal control problems whose necessary optimality conditions apply in the normal form, under some constraint qualifications. The main purpose of introducing such problems is that calculus of variations problems can be regarded as an optimal control problem (owing to the `state augmentation' procedure). Therefore, the results on optimal control problems will be used to establish normality of optimality conditions for the reference calculus of variations problem (\ref{problem: calculus of variations}).
\vskip2ex

\subsection{Normality in Optimal Control Problems}
We recall that `normality' means that the Lagrangian multiplier associated with the objective function -- here written $\lambda$ -- is different from zero (it can be taken equal to 1).\\
Consider the fixed left end-point optimal control problem (\ref{problem: ocp with state constraint chap 5}) with a state constraint set $A \subset \R^n$ which is merely a closet set:

\begin{equation}\begin{cases}\label{problem: ocp with state constraint chap 5}
		\begin{aligned} 
			& {\text{minimize}}
			&& g( x(T))  \\
			&&& \hspace{-1.9cm} \text{over } x\in W^{1,1}([S,T],\R^n) \text{ and measurable functions } u \text{ satisfying}  \\
			&&& \dot{x}(t) = f(t,x(t), u(t)) \quad \textrm{ a.e. } t \in [S,T] \\
			&&& x(S)=x_{0} \\
			&&&  x(t) \in A \quad \text{ for all } t \in [S,T] \\
			&&& u(t) \in U(t) \quad \textrm{a.e. } t \in [S,T] \  . \end{aligned}\end{cases}\tag{P}\leqnomode
\end{equation}
\ \\
\noindent
The data for this problem comprise functions $g: \mathbb{R}^n \longrightarrow \mathbb{R}, f:[S,T] \times \mathbb{R}^n \times \mathbb{R}^m \longrightarrow \mathbb{R}^n,$ an initial state  $x_0 \in \mathbb{R}^n,\text{ and a multifunction } U(.): [S,T] \leadsto \mathbb{R}^m.$ The set of control functions for (\ref{problem: ocp with state constraint chap 5}), denoted by $\mathcal{U}$, is the set of all measurable functions $u: [S,T] \longrightarrow \mathbb{R}^m \text{ such that } u(t) \in U(t) \text{ a.e. } t\in [S,T].$
% A pair of functions $(x, u)$ comprising an absolutely continuous function $x \in W^{1,1}([S,T];\mathbb{R}^n)$ and $u \in \mathcal{U}$, for which the differential equation of problem $(P)$ is satisfied, is called a \textit{process}. The first component of a process is called a \textit{state trajectory}. A process for which all the constraints of problem $(P)$ are satisfied is called an \textit{admissible process}. A \textit{ minimizer } is an admissible process which achieves the minimum of $g(x(T))$ over all admissible processes $(x, u)$.
\ \\
\noindent
We say that an admissible process $(\bar{x}, \bar{u})$ is a $W^{1,1}-$\textit{local minimizer} if there exists $\epsilon >0$ such that \[ g(\bar{x}(T)) \leq g(x(T)),\] for all admissible processes $(x,u)$ satisfying \[ \| x(.)-\bar{x}(.) \| _{W^{1,1}} \leq \epsilon. \]
\noindent
There follows a `normal' version of the maximum principle for state constrained problems. For a $W^{1,1}-$local minimizer $(\bar{x}, \bar{u})$ and a positive scalar $\delta$, we assume the following:
\begin{enumerate}[label= (H\arabic*), ref=H\arabic*]
	\item \label{item: H1 ocp chap5} The function $(t,u) \mapsto f(t,x,u)$ is measurable for each $x\in \mathbb{R}^n$. There exists a measurable function $k(t,u)$ such that $t \mapsto k(t,\bar{u}(t))$ is integrable and \[ | f(t,x,u)- f(t,x',u) | \leq k(t,u) | x-x' | \] for $ x, x' \in \bar{x}(t) + \delta \mathbb{B}, \ u \in U(t), {\rm a.e.}\ t\in [S,T]$. Furthermore there exist scalars $K_f >0$ and $\varepsilon' >0 $ such that \[ | f(t,x,u)- f(t,x',u) | \leq K_f | x-x' | \] for $x, x' \in \bar{x}(S) + \delta \mathbb{B}, \ u \in U(t), \text{ a.e. } t\in [S,S+\varepsilon'].$
	\item  \label{item: H2 ocp chap5} Gr $U(.)$ is measurable.
	\item\label{item: H3 ocp chap5} The function $g$ is Lipschitz continuous on $\bar{x}(T) + \delta \mathbb{B}$.
	
\end{enumerate}
Reference is also made to the following constraint qualifications. There exist positive constants $K, \tilde \varepsilon, \tilde \beta, \tilde \rho$ and a control $\hat{u} \in \mathcal{U}$ such that
\begin{enumerate}[label= (CQ\arabic*), ref= CQ\arabic*]
	\item \label{item: CQ1 normality ocp chap 5}
	\begin{equation} \label{boundedness boundary point}
	| f(t,\bar{x}(t), \bar{u}(t))-f(t,\bar{x}(t),\hat{u}(t)) | \leq K, \quad \text{ for a.e. } t \in (\tau- \tilde \varepsilon, \tau] \cap [S,T]
	\end{equation}
	and
	\[
	\eta \cdot [f(t,\bar{x}(t), \hat{u}(t))-f(t,\bar{x}(t),\bar{u}(t))] < -\tilde \beta,
	\]
	for all $\eta \in \partial^{>}_x d_A(\bar x (s)), \text{ a.e. }  s, \; t \in (\tau - \tilde \varepsilon, \tau] \cap [S,T] $ and for all $\tau \in \lbrace \sigma \in [S,T]\; : \; \bar{x}(\sigma) \in \partial A \rbrace.$
	\item\label{item: CQ2 normality ocp chap 5} If $ x_0 \in \partial A,$ then for a.e. $t \in [S, S+ \tilde \varepsilon)$
	\begin{equation} \label{boundedness initial point}
	| f(t,x_0,\hat{u}(t)) | \leq K,  \quad | f(t,x_0,\bar{u}(t)) | \leq K, 
	\end{equation}
	and
	\[
	\eta \cdot [f(t,x_0, \hat{u}(t))-f(t,x_0,\bar{u}(t))] < - \tilde \beta \]
	for all $\eta \in  \partial^{>}_x d_A(x), \; x \in (x_0 + \tilde \rho \B) \cap \partial A .$
	\item\label{item: CQ3 normality ocp chap 5} \index{Pointed Convex Cone}
	\[
	{  \rm co }\ N_{A}(\bar{x}(t))  \text{ is pointed for each } t \in [S,T].
	\]
	
\end{enumerate}
\begin{theorem}\label{Thm3}
	Let $(\bar{x},\bar{u})$ be a $W^{1,1}-$local minimizer for (\ref{problem: ocp with state constraint chap 5}). Assume that hypotheses (\ref{item: H1 ocp chap5})-(\ref{item: H3 ocp chap5}) and the constraint qualifications (\ref{item: CQ1 normality ocp chap 5})-(\ref{item: CQ3 normality ocp chap 5}) hold. Then, there exist $p(.) \in W^{1,1}([S,T], \mathbb{R}^n)$, a Borel measure $\mu(.)$ and a $\mu-$integrable function $\gamma(.)$ such that
	\begin{enumerate}[label= (\roman{*}), ref= \roman{*}]
		\item $	- \dot{p}(t) \in {\rm co }\ \partial_{x} (q(t) \cdot f(t,\bar{x}(t), \bar{u}(t)) \; \; {\rm a.e.}\ t\in [S,T],$
		\item $	-q(T) \in \partial g (\bar{x}(T)),$ 
		\item $q(t) \cdot f(t,\bar{x}(t),\bar{u}(t)) = \max_{u\in U(t)} q(t) \cdot f(t,\bar{x}(t),u),$
		\item $	\gamma(t) \in  \partial^{>}d_A(\bar{x} (t))$ and $\textrm{supp} (\mu) \subset \lbrace t\in [S,T]\; : \; \bar{x}(t) \in \partial A \rbrace,$
	\end{enumerate}
	where
	\begin{equation*} q(t) = \begin{cases}
			p(S)   & \quad  t =S   \\
			p(t) + \int_{[S,t]} \gamma(s) d\mu(s)   & \quad t \in (S,T].
		\end{cases} \end{equation*}
	\end{theorem}
	\begin{proof}
		
	This result was proved in \cite{fontes_normal_2013} for $L^{\infty}-$local minimizers and for a state constraint expressed as an inequality function. It remains valid for the weaker case of $W^{1,1}-$local minimizers (by adding eventually an extra variable $\dot{y}(t) = |f(t,x(t),u(t))-\dot{\bar{x}}(t)|$) and for a state constraint expressed in terms of a closed set $A$.
	
	\noindent
	Indeed, let $(\bar x,\bar u)$ be merely a $W^{1,1}-$local minimizer for (\ref{problem: ocp with state constraint chap 5}). Then there exists $\alpha>0$ such that $(\bar x, \bar y \equiv 0, \bar u)$ is a $L^\infty-$local minimizer for
	\begin{equation*}\begin{cases}
	\begin{aligned} \label{augmented problem}
	& {\text{minimize}}
	&& g( x(T))  \\
	&&& \hspace{-1.9cm} \text{over } x\in W^{1,1}([S,T],\R^n) \text{ and measurable functions } u \text{ satisfying}  \\
	&&& \dot{x}(t) = f(t,x(t), u(t)) \quad \textrm{ a.e. } t \in [S,T] \\
	&&& \dot{y}(t) = |f(t,x(t),u(t))-\dot{\bar{x}}(t)| \quad \textrm{ a.e. } t \in [S,T] \\
	&&& (x(S),y(S), y(T)) \in \{x_{0}\} \times \{ 0 \} \times \alpha\B \\
	&&&  x(t) \in A \quad \text{ for all } t \in [S,T] \\
	&&& u(t) \in U(t) \quad \textrm{a.e. } t \in [S,T] \  . \end{aligned}\end{cases}\tag{P1}\leqnomode
	\end{equation*}
	(It suffices to prove it by contradiction.) Therefore we can apply \cite[Theorem 4.2]{fontes_normal_2013} in its normal formulation to the augmented problem (\ref{augmented problem}) with reference to the $L^\infty-$local minimizer $(\bar x, \bar y \equiv 0, \bar u)$. After expliciting the corresponding necessary optimality conditions, we notice that the adjoint arc $r$ associated with the state variable $y$ is constant because the right hand side of the dynamics does not depend on $y$. Moreover, $r=0$ owing to the transversality condition since $\bar y(T) \in \alpha \B$ is inactive. Hence, we deduce the required necessary optimality conditions as represented in Theorem \ref{Thm3} covering the case of $W^{1,1}-$local minimizers.
	\end{proof}
	\begin{remark} \label{remark1}
		Normality of the necessary optimality conditions for state constrained optimal control problems has been well studied and many results exist in the literature where each result requires some regularity assumptions on the problem data (cf. \cite{fontes_normal_2013}, \cite{frankowska2009normality}, \cite{frankowska2013inward}, \cite{fontes_normality_2015} etc.). In \cite{fontes_normal_2013}, normality is established for free right-endpoint optimal control problems having $L^{\infty}-$local minimizers. The constraint qualifications assumed on the data are in the form of (\ref{item: CQ1 normality ocp chap 5}) and (\ref{item: CQ2 normality ocp chap 5}) but considered for an inequality state constraint of a possibly nonsmooth function. This result was extended in \cite{fontes_normality_2015} (considering always the case of $L^\infty-$local minimizers) to cover a larger class of problems where the state constraint is given in terms of a closed set and the right-endpoint (denoted by the set $K_1$ in \cite{fontes_normality_2015}) is considered to be a closed subset of $\R^n$. New (and weaker) constraint qualifications are discussed in \cite{fontes_normality_2015} to guarantee normality when $\bar{x}(T) \in \text{int }K_1$: the inward pointing condition has just to be satisfied for almost all times at which the optimal trajectory has an outward pointing velocity. Moreover in \cite{fontes_normality_2015} relations with previous constraint qualifications are examined.
		
		\noindent 
		Theorem \ref{Thm3} might seem to be a particular case of \cite[Theorem 3.2]{fontes_normality_2015} when $K_1 = \R^n$ (eventually after considering the augmented problem in order to cover the case of $W^{1,1}-$local minimizers). However, looking closely to assumption (H3) in \cite{fontes_normality_2015}, we see that (H3) is a stronger assumption than the one considered in previous papers (see for instance \cite{ferreira_nondegenerate_1999}, \cite{lopes_constraint_2011}, \cite{fontes_normal_2013}) and in assumptions (\ref{boundedness boundary point}) and (\ref{boundedness initial point}) considered in our paper. Indeed, in (\ref{boundedness boundary point}) and (\ref{boundedness initial point}), the boundedness of the dynamics is considered only for the optimal control $\bar u(.)$ and the constructed control $\hat u(.)$ at the initial data $x_0$ and at the optimal trajectory points $\bar x(t)$ only for times $t < \tau$ at which the trajectory hits the boundary of the state constraint, however in \cite{fontes_normality_2015}, the boundedness concerns all controls $u \in U(t)$ and all $x$ near the optimal trajectory $\bar x(.)$, which is not implied by the assumptions of our paper.
		
		\noindent
		We underline the fact that to derive the desired necessary conditions stated in Theorem \ref{Thm1} above, a crucial requirement is the possibility to apply normality results for optimal control problems with dynamics which might be unbounded (cf. Remark \ref{remark2}).
	\end{remark}

	\subsection{Technical Lemmas}
	
	We invoke now two technical lemmas which are crucial for establishing the proof of Theorem \ref{Thm1}.
	%\begin{lemma}{(cf. \cite[Theorem 11.5.1]{vinter_optimal_2010})}\label{lemma: lipschitz conituity minmizer calculus of variations}
	%	Let $\bar{x}$ be a $W^{1,1}-$local minimizer for (\ref{problem: calculus of variations}) for which we assume that hypotheses (\ref{item: CV1})-(\ref{item: CV3}) are satisfied, then $\bar{x}$ is a Lipschitz continuous function.
	%\end{lemma}
	The first lemma says that one can select a particular bounded control $v(.)$ which pulls the dynamics inward the state constraint set more than the reference minimizer.
	
	\begin{lemma} \label{Lemma3} Let $\bar{x}(.)$ be a Lipschitz $W^{1,1}-$local minimizer for (\ref{problem: calculus of variations}) for which (\ref{item: CV1}) and (\ref{CQ for calculus of variations}) are satisfied. Then, there exist positive constants $\varepsilon, \; \rho, \; \beta, \; C, \; C_1$ and a measurable function $v(.)$ such that:
		\begin{enumerate}[label=(\roman{*}), ref=\roman{*}]
			\item\label{item: control existence bounded chapter CV} $ \| v -\dot{\bar{x}} \| _{L^{\infty}} \leq C$ and $\| v \| _{L^{\infty}} \leq C_{1}.$
			\item \label{item: control existence ipc1 chapter CV}For all $ \tau \in \lbrace \sigma \in [S,T] : \bar{x}(\sigma) \in \partial A \rbrace$,
			\[
			\mathop {\sup }\limits_{\begin{array}{*{20}c}
				{\eta  \in {\rm co} \left(N_A(\bar{x}(s)) \cap  \partial \mathbb{B}\right) }  \\
				\end{array}}   (v(t)-\dot{\bar{x}}(t)) \cdot \eta < -\beta, \quad {\rm a.e. } \ s,t \in (\tau-\varepsilon, \tau].
			\]
			\item\label{item: control existence ipc2 chapter CV} If $x_0=\bar{x}(S) \in \partial A$, then
			\[
			\mathop {\sup }\limits_{\begin{array}{*{20}c}
				{\eta  \in {\rm co} \left(N_A(x) \cap  \partial \mathbb{B}\right) } \\
				{x \in (x_0 + \rho {\mathbb{B}}) \cap \partial A.}  \\
				\end{array}}
			(v(t)- \dot{\bar{x}}(t))\cdot \eta < - \beta, \quad {\rm a.e. } \ t \in [S,S+\varepsilon). \]
		\end{enumerate}
	\end{lemma}
	
	For the proof of Lemma \ref{Lemma3}, we shall invoke the following technical lemma:
	\begin{lemma} \label{Lemma2} Fix $R>0$. Assume that	${\text{ int }}\ T_A(z) \neq \emptyset$ for all $z \in \partial A \cap (R+1) \B$. Then we can find positive numbers $ \beta, ~\epsilon_{0}, ~\epsilon_{1},\ldots, \epsilon_{k}$, points $z_{1}, \ldots, z_{k} \in \partial{A} \cap (R+1)\mathbb{B}$, and vectors $\zeta_{j} \in {\rm int }\ T_{A}(z_{j}), \text{ for } j=1,\ldots, k,$ \text{such that}
		\begin{enumerate}[label = (\roman{*}), ref= \roman{*}]
			\item \label{item: i lemma appendix chap 5} $ \mathop  \bigcup \limits_{j = 1}^k  (z_j+ \frac{\epsilon_j}{2} {\rm int }\ \mathbb{B})  \supset (\partial A + \epsilon_{0} \mathbb{B}) \cap (R+1)\mathbb{B}$ ,
			\item \label{item: ii lemma appendix chap 5} \[ \mathop {\sup }\limits_{\begin{array}{*{20}c}
				{\eta  \in { \rm co }(N_A(z) \cap \partial \mathbb{B})}  \\
				{z \in (z_j + \epsilon_{j} {\mathbb{B}}) \cap \partial A.}  \\
				\end{array}}
			\zeta_j \cdot \eta < -\beta \ , \quad \text{for all } j=1, \ldots, k \ . \] \end{enumerate}
	\end{lemma}

	\begin{proof} [\textbf{Proof of Lemma \ref{Lemma3}}] 
		
	 Indeed, by defining the function $v(.): [S,T] \rightarrow \mathbb{R}^{n}$ as follows
		\[ v(t) := \dot{\bar{x}}(t) + \sum \limits_{j=1}^{k} \chi_{z_{j}+ \frac{\epsilon_{j}}{2} {\rm int }\ \mathbb{B}} (\bar{x}(t)) \; \zeta_{j}, \quad \quad t \in [S,T] \]
		where $z_j$, $\epsilon_j$ and $\zeta_j$ are respectively the points, positive numbers and vectors of Lemma \ref{Lemma2}. $ \chi_{Y}$ denotes the characteristic function of the subset $ Y \subset \mathbb{R}^{n}$, and owing to the assertions of Lemma \ref{Lemma2},  we can prove that the constructed $v(.)$ verifies the statement of Lemma \ref{Lemma3}.	(We refer the reader to \cite{khalil:tel-01740334}  for full details of the proof.)
		
	\end{proof}
	
	\subsection{Proof of Theorem \ref{Thm1}}

		We first employ a standard argument, the so-called `state augmentation', which allows to write the problem of calculus of variations (\ref{problem: calculus of variations}) as an optimal control problem of type (\ref{problem: ocp with state constraint chap 5}).	Indeed, it is enough to add an extra absolutely continuous state variable \[ \index{State Augmentation}z(t)= \int_{S}^{t} L(s,x(s), \dot{x}(s))ds\]
		and consider the dynamics $\dot{x} =u$. We notice that $z(S)=0$ and $U(t)= \mathbb{R}^n$.
		\vskip2ex
		\noindent
		Then, the problem (\ref{problem: calculus of variations}) can be written as the optimal control problem (\ref{ocp by state augmentation chap 5}):
		
		\begin{equation}\begin{cases}
				\begin{aligned} \label{ocp by state augmentation chap 5}
					& {\text{minimize}}
					&& z(T)  \\
					&&& \hspace{-1.9cm} \text{over } W^{1,1}-\text{ arcs } (x(.),z(.)) \text{ and measurable functions } u(.) \in \R^n \text{ satisfying}  \\
					&&& (\dot{x}(t), \dot z(t)) = (u(t),L(t,x(t),u(t))) \quad \textrm{ a.e. } t \in [S,T] \\
					&&& (x(S), z(S)) = (x_{0},0) \\
					&&&  x(t) \in A \quad \text{ for all } t \in [S,T] \\
					&&& u(t) \in U(t)=\R^n \quad \textrm{a.e. } t \in [S,T] \  . \end{aligned}\end{cases}\tag{P$'$}\leqnomode
		\end{equation}
		
		We set $w(t):=
		\begin{pmatrix}
		x(t) \\
		z(t)
		\end{pmatrix}$ and $\tilde{f}{(t,w(t), u(t))} :=
		\begin{pmatrix}
		u(t) \\ L(t,x(t),u(t))
		\end{pmatrix}.$
		Here the set of controls $\mathcal{U}$ is the set of dynamics $\dot{x}(t) \in \mathbb{R}^{n}$ for a.e. $t \in [S,T]$.

		It is easy to prove that if $\bar x$ is a $W^{1,1}-$local minimizer for the reference calculus of variations problem (\ref{problem: calculus of variations}), then  $\left(\bar{w} (t)= \begin{pmatrix}
		\bar{x}(t) \\
		\bar{z}(t)
		\end{pmatrix}, \bar{u} \right)$ is a $W^{1,1}-$local minimizer for (\ref{ocp by state augmentation chap 5}) where $\dot{\bar{z}}(t)= L(t,\bar{x}(t), \bar{u}(t))$ and $\bar{u}= \dot{\bar{x}}$.

		\vskip2ex
		\noindent
		The proof of Theorem \ref{Thm1} is given in three steps. The first step is devoted to show that the constraint qualifications (\ref{item: CQ1 normality ocp chap 5})-(\ref{item: CQ3 normality ocp chap 5}) of Theorem \ref{Thm3} are mainly implied by (\ref{CQ for calculus of variations}) (of Theorem \ref{Thm1}). In step 2, we verify that hypotheses (\ref{item: H1 ocp chap5})-(\ref{item: H3 ocp chap5}) of Theorem \ref{Thm3} can be deduced from hypothesis (\ref{item: CV1}) of  Theorem \ref{Thm1}. In step 3, we apply Theorem \ref{Thm3} to (\ref{ocp by state augmentation chap 5}) in order to obtain the assertions of Theorem \ref{Thm1}.
		\vskip2ex
		\noindent
		{\bf Step 1.} Prove that the constraint qualifications (\ref{item: CQ1 normality ocp chap 5})-(\ref{item: CQ3 normality ocp chap 5}) of Theorem \ref{Thm3} are mainly implied by (\ref{CQ for calculus of variations}) (of Theorem \ref{Thm1}).
		\vskip1ex
		\noindent
		The constraint qualifications (\ref{item: CQ1 normality ocp chap 5}) and (\ref{item: CQ2 normality ocp chap 5}) for the optimal problem (\ref{ocp by state augmentation chap 5}) become as follows: there exist positive constants $K, \tilde \varepsilon, \tilde \beta, \tilde \rho$ and a control $\hat{u} \in \mathcal{U}$ such that
		\begin{enumerate}[label= (CQ\arabic*)$'$ , ref= (CQ\arabic*)$'$]
			\item \label{CQ1'}
			\begin{eqnarray} \label{8}
				| \tilde{f}{(t,(\bar{x}(t),\bar{z}(t)),\bar{u}(t))}-\tilde{f}{(t,(\bar{x}(t),\bar{z}(t)),\hat{u}(t))} | \leq K, \; \text{a.e. } t \in (\tau- \tilde \varepsilon, \tau] \cap [S,T]
			\end{eqnarray}
			and
			\begin{equation} \label{9}
				\tilde{\eta} \cdot [\tilde{f}{(t,(\bar{x}(t),\bar{z}(t)), \hat{u}(t))}-\tilde{f}{(t,(\bar{x}(t),\bar{z}(t)),\bar{u}(t))}] < -\tilde \beta,
			\end{equation}
			for all $\tilde{\eta} \in \partial^{>}_{w} d_{A}(\bar{x}(s)) \text{ a.e.}\ s, \; t \in (\tau - \tilde \varepsilon, \tau] \cap [S,T] \text{ and for all } \tau \in \lbrace \sigma \in [S,T]\; : \; \bar{x}(\sigma) \in \partial A \rbrace$. Here	$
			\partial_w^> d_A(\xb(s)) :={\rm co\,} \{  (a,b): \text{there exists } \; s_i \rightarrow s \text{ such that } d_A(\xb(s_i))>0 \, \,  \text{for all } i \ , d_A(\xb(s_i)) \rightarrow d_A(\xb(s))\text{ and } \nabla_{w}d_A (\xb(s_i)) \rightarrow(a,b) \}.
			$
			\item \label{CQ2'} If $ x_0 \in \partial A$ then for a.e. $ t \in [S, S+ \tilde \varepsilon)$
			\begin{equation} \label{5}
				| \tilde{f}{(t,(x_0,0),\hat{u}(t))} | \leq K \quad \text{ and } \quad
				| \tilde{f}{(t,(x_0,0),\bar{u}(t))} | \leq K,
			\end{equation}
			and  \begin{equation} \label{7} \tilde{\eta} \cdot [\tilde{f}{(t,(x_0,0), \hat{u}(t))}-\tilde{f}{(t,(x_0,0),\bar{u}(t))}] < -\tilde \beta,
			\end{equation} for all $\tilde{\eta} \in \partial^{>}_{w} d_{A} (x), \; x \in (x_0 + \tilde \rho \B) \cap \partial A$. Here,
			$
			\partial_w^> d_A(x)={\rm co\,} \{  (a,b):\text{there exists } \; x_i \xrightarrow{d_A} x \text{ such that } d_A(x_i)>0 \, \,  \text{for all } i
			\; \text{ and } \; \nabla_{w}d_A (x_i) \rightarrow(a,b) \}.
			$
		\end{enumerate}
		\vskip2ex
		\noindent
		{\bf 1.} We start with the proof of condition (\ref{5}).
		
		\noindent
		By the Lipschitz continuity of the minimizer $\bar{x}$, we take $R:=| x_0 | + \| \dot{\bar{x}} \|_{L^{\infty}} (1+T-S) >0$. Then $\bar{x}([S,T])\subset R \mathbb{B}$ and $| \dot{\bar{x}}(t) | \leq R$ for a.e. $t \in [S,T].$ Take $x_0 \in \partial A$. Since the function $L(.,.,.)$ is bounded on bounded sets (owing to (\ref{item: CV1})), we obtain
		\[
		| \tilde{f}{(t,(x_0,0), \bar{u}(t))} | = \left| \begin{pmatrix}
		\bar{u}(t)= \dot{\bar{x}}(t) \\
		L(t,x_0,\bar{u}(t))
		\end{pmatrix} \right| \leq K \quad \text{ a.e. } t \in [S,T], \ \text{ for some } K > 0.
		\]
		Moreover, under (\ref{CQ for calculus of variations}), Lemma \ref{Lemma3} ensures the existence of a measurable function $v(.)$ and positive constants $C$ and $C_1$ such that $\| v \|_ {L^{\infty}} \leq C_1$ and $\| v - \dot{\bar{x}} \|_ {L^{\infty}} \leq C.$ Therefore, by choosing a control $\hat u$ such that $\hat u(t)= v(t)$ a.e. $t$, we deduce easily that
		\[ | \tilde{f}(t,(x_0,0),\hat u(t)) | \le K \quad \text{a.e. } t \in [S,T]  \quad \text{for some } K>0 . \] Condition (\ref{5}) is therefore satisfied.
		
		\vskip2ex
		\noindent
		{\bf 2.} Condition (\ref{8}) follows also (for $K$ big enough) from the particular choice of $\hat u (t)$ to be the bounded control $v(.)$ of Lemma \ref{Lemma3}, and from (\ref{item: CV1}).

%		Now we invoke Lemma \ref{Lemma3} and we consider the control function $\hat u : [S,S+\varepsilon] \rightarrow \R$ to be defined as $\hat{u}(t):= v(t)$, where $v(t)$ is a measurable function which satisfies the assertions of Lemma \ref{Lemma3} and $\varepsilon$ is the positive constant of Lemma \ref{Lemma3}. Regarding this choice, we deduce easily that
%		\[ | \tilde{f}(t,(x_0,0),\hat u(t)) | \le K \quad \text{a.e. } t \in [S,S+\varepsilon)  \quad \text{for some } K>0 ,\] and in particular for a.e. $t\in [S,S+\tilde{\varepsilon})$ by choosing $\tilde{\varepsilon} \in (0, \varepsilon]$.
% Condition (\ref{5}) is therefore satisfied.
 \vskip2ex
 \noindent
 {\bf 3.} To prove condition (\ref{7}), we consider the same choice of the control function $\hat{u}(.)$, that is $\hat u(t):=v(t)$ for $t \in [S,S+\varepsilon]$, where $\varepsilon$ is the positive number and $v(t)$ is the measurable function which satisfy the properties of Lemma \ref{Lemma3}. Condition (iii) of Lemma \ref{Lemma3} implies that there exist positive constants $ \rho, \beta $ such that 
		\begin{equation}\label{equation: choice of control proof main theorem}
		\eta \cdot(\hat{u}(t)-\dot{\bar x}(t)) < -\beta   \quad \text{for a.e.  } t \in [S,S+\varepsilon)
		\end{equation}
		for all $\eta \in {\rm co }(N_A(x) \cap \partial \mathbb{B})$ and for all $x \in (x_0 + \rho \mathbb{B}) \cap \partial A.$ In particular for all $\eta \in \partial^>_x d_A(x)$ (owing to Proposition \ref{Prop2}(i)) and by choosing $\tilde \varepsilon \in (0, \varepsilon] $, \, $ \tilde \rho \in (0, \rho] $ and $\tilde \beta = \beta.$ Take now any $\tilde \eta \in \partial_w^> d_A(x)$ where $x \in (x_0 + \tilde \rho \B) \cap \partial A$. By definition of $\partial^{>}_{w} d_{A}(x)$
		\[
		\begin{array}{l}
		\tilde{\eta} \in  \Big(\mathrm{co\,}\{a: \text{there exists } \; x_i \xrightarrow{d_A} x \text{ such that }
		\,  d_A(x_i)>0 \, \,
		\text{for all } i,\, \text{ and } \nabla_x d_A(x_i)\rightarrow a\},0\Big).
		\end{array}
		\] This is because $\nabla_z d_A(x_i)=0$. We conclude that $\tilde{\eta}$ can be written as:
		\begin{equation*}
		\tilde{\eta} :=(\eta,0) \quad \text{ where }\eta\in \partial^> d_A(x) \ .
		\end{equation*}
		It follows that, owing to (\ref{equation: choice of control proof main theorem}):
		\[
		(\eta,0) \cdot [\tilde{f}{(t,(x_0,0), \hat{u}(t))}-\tilde{f}{(t,(x_0,0),\bar{u}(t))}] = \eta \cdot (\hat{u}(t)-\bar{u}(t)) < - \tilde \beta,
		\] for all $\eta \in \partial ^{>} d _{A}(x)$ a.e. $t \in [S,S+\tilde \varepsilon)$ for all $x \in (x_0+\tilde \rho \B) \cap \partial A$. Condition (\ref{7}) is therefore confirmed.

		\vskip2ex
		\noindent
		{\bf 4.}
		For the proof of (\ref{9}), we take any $\tau \in \lbrace \sigma \in [S,T]\; : \; \bar{x}(\sigma) \in \partial A \rbrace$. Consider again the positive constants $\varepsilon$ and $\beta$ and the selection $v(.)$
		provided by Lemma \ref{Lemma3}. Property (\ref{item: control existence ipc1 chapter CV}) of the latter lemma implies that for a.e. $s, \ t \in (\tau-\varepsilon, \tau]$
		\begin{equation}  \label{choice of control 2 proof main theorem chap 5} \eta \cdot (v(t) - \dot{\bar x} (t)) < -\beta \end{equation} for all $\eta \in {\rm co}(N_{A}(\bar x (s)) \cap \partial \B).$ In particular for all $\eta \in \partial^{>}_{x} d_{A}(\bar x (s))$ (owing to Proposition \ref{Prop2}(i)) and for $\tilde \varepsilon \in (0, \varepsilon] $ and $ \tilde \beta = \beta$ and by considering the control function $\hat u (t) := v(t)$ for $t\in (\tau - \varepsilon, \tau]$ (and eventually for $t\in (\tau-\tilde{\varepsilon}, \tau]$). Now we take an element $\tilde \eta$ in $\partial^>_{w} d_A(\xb(s))$ for $ s \in (\tau -\tilde \varepsilon, \tau] \cap [S,T]$.
		\noindent
		It follows that,
		\[
		\begin{array}{l}
		\tilde{\eta} \in  \Big(\mathrm{co\,}\{a: \text{there exists } \, t_i \rightarrow s\text{ s.t. }
		\,  d_A(\xb(s_i))>0 \, \,
		\text{for all } i,\, d_A(\xb(s_i)) \rightarrow d_A(\xb(s))\text{ and } \\ \hspace{3 in}\nabla_x d_A(\xb(s_i))\rightarrow a \},0 \Big)
		\end{array}
		\]
		which is equivalent to write $\tilde{\eta}$ as
		\begin{equation*}
			\tilde{\eta} :=(\eta,0) \quad  \text{where }\eta\in \partial^> d_A(\xb(s))
		\end{equation*}
		and $ s \in (\tau - \tilde \varepsilon, \tau] \cap [S,T]$. Making use of (\ref{choice of control 2 proof main theorem chap 5}), it follows that
		\[
		(\eta,0) \cdot [\tilde{f}{(t,(\bar{x}(t),\bar{z}(t)), \hat{u}(t))}-\tilde{f}{(t,(\bar{x}(t),\bar{z}(t)),\bar{u}(t))}]= \eta \cdot (\hat{u}(t)-\bar{u}(t)) < - \tilde \beta,
		\]
		for all $\eta \in \partial^>d_A(\bar{x}(s))$ for a.e. $s, t \in (\tau - \tilde \varepsilon,\tau] \cap [S,T]$ and for all $\tau \in \lbrace \sigma \in [S,T] \; : \; \bar{x}(\sigma) \in \partial A \rbrace.$ We conclude that condition (\ref{9}) is satisfied.
		
		\vskip2ex
		\noindent
		{\bf 5.} Finally, it is clear that  (\ref{CQ for calculus of variations}) implies that (\ref{item: CQ3 normality ocp chap 5}) is satisfied.
		\vskip4ex
		\noindent
		\textbf{Step 2.} It is an easy task to prove that hypotheses (\ref{item: H1 ocp chap5})-(\ref{item: H3 ocp chap5}) adapted to the optimal control problem (\ref{ocp by state augmentation chap 5}) are satisfied. This is a direct consequence of hypothesis (\ref{item: CV1}).
%		\begin{itemize}
%			\item Owing to (\ref{item: CV1}), it is straightforward to see that the map $(t,u) \mapsto \tilde{f}{(t,(x,z),u)} =  \begin{pmatrix}
%			u  \\
%			L(t,x,u) \end{pmatrix} $ is a measurable function, for each $x\in \R^n$. Moreover, since $L(t,.,u)$ is Lipschitz for all  $(t,u)\in [S,T] \times \R^n$ (cf. (\ref{item: CV1})), we deduce that
%			\[
%			| \tilde{f}(t,(x,z),u)- \tilde{f}(t,(x',z'),u) | = \left| \begin{pmatrix}
%			0  \\
%			L(t,x,u)-L(t,x',u)
%			\end{pmatrix} \right| \leq K_L | x-x' |, \] for all $x, x' \in  \bar{x}(t) + \epsilon' \mathbb{B},$  $t\in [S,T]$ and $u \in \mathbb{R}^n.$ The Lipschitz continuity w.r.t. $z$ is obvious since $(x,z) \mapsto \tilde{f}{((x,z),u)}$ is independent on $z$.
%			Hypothesis (\ref{item: H1 ocp chap5}) is therefore satisfied by choosing $\delta= \epsilon'$. 
%			\item $U(t)= \mathbb{R}^n, \; U(.): [S,T] \leadsto \mathbb{R}^n,$ a multifunction. By consequence, ${\rm Gr}\ U(.)= [S,T] \times \mathbb{R}^n$ which is a measurable set. Hypothesis (\ref{item: H2 ocp chap5}) is satisfied.
%			\item Hypothesis (\ref{item: H3 ocp chap5}) is verified since the cost function of problem (\ref{ocp by state augmentation chap 5}), $g(x,z):=z$, is Lipschitz.
%		\end{itemize}
%		
		\vskip4ex
		\noindent
		\textbf{Step 3.} Apply Theorem \ref{Thm3} to (\ref{ocp by state augmentation chap 5}) and then obtain the assertions of Theorem \ref{Thm1}. 
		\vskip1ex
		\noindent
		Since $\bar x$ is a $W^{1,1}-$local minimizer for the reference calculus of variations problem (\ref{problem: calculus of variations}), then $((\bar x, \bar z),\dot{\bar x}=\bar u)$ is a $W^{1,1}-$local minimizer for the optimal control problem (\ref{ocp by state augmentation chap 5}). Theorem \ref{Thm3} can be therefore applied to the problem (\ref{ocp by state augmentation chap 5}). Namely, there exist a couple of absolutely continuous functions $(p_1,p_2) \in W^{1,1}([S,T],\mathbb{R}^n) \times W^{1,1}([S,T],\mathbb{R})$, a Borel measure $\mu(.)$ and a $\mu-$integrable function $\gamma(.)$, such that:
		\begin{enumerate}[label = (\roman{*})$'$, ref= (\roman{*})$'$]
			\item \label{item: adjoint arc proof main thm chap 5} $-(\dot{p_1}(t), \dot{p_2}(t)) \in \text{ co } \partial_{(x,z)} ((q_1(t),q_2(t)) \cdot \tilde{f}(t,(\bar{x}(t),\bar{z}(t)),\bar{u}(t))) \text{ a.e. } t \in [S,T],$
			\item \label{item: transversality condition proof main thm chap 5}$-(q_1(T), q_2(T)) \in \partial_{(x,z)}  \bar{z}(T),$
			\item \label{item: weirsterass condition proof main thm chap 5}$(q_1(t), q_2(t)) \cdot \tilde{f}(t,(\bar{x}(t),\bar{z}(t)), \bar u(t))= \max_{u \in U	(t)} (q_1(t), q_2(t)) \cdot \tilde{f}(t,(\bar{x}(t),\bar{z}(t)), u)  ,$
			\item \label{item: element in the hybrid subdifferential proof main thm chap 5}$\gamma(t) \in \partial^> d_A(\bar{x}(t))$ \quad and \quad $\textrm{supp}(\mu) \subset \lbrace t\in [S,T]\; : \; \bar{x}(t) \in \partial A \rbrace,$
		\end{enumerate}
		where
		\begin{equation*}
			q_1(t)= \begin{cases}
				p_1(S) & \quad  t =S  \\
				p_1(t)+ \int_{[S,t]} \gamma(s) d\mu(s), & \quad t\in (S,T]  \end{cases}
		\end{equation*}
		and
		\[
		q_2(t)= p_2(t), \quad \text{for } t \in [S,T].
		\]	
		The transversality condition \ref{item: transversality condition proof main thm chap 5} ensures that $q_1(T)=0$. This implies that condition (\ref{2_chap5}) of Theorem \ref{Thm1} is satisfied. Furthermore, $q_2(T)=-1 =p_2(T)$. 
		\noindent
		The adjoint system \ref{item: adjoint arc proof main thm chap 5} ensures, by expanding it and applying a well-known nonsmooth calculus rule, that:	
		\[ -(\dot{p_1}(t), \dot{p_2}(t)) \in \textrm{co } \partial_{(x,z)} \left(q_1(t) \cdot \bar{u}(t) + q_2(t) L(t,\bar{x}(t), \bar{u}(t)) \right).\]
		By the Lipschitz continuous of  $L(t,.,u)$ on a neighborhood of $\bar{x}(t)$, we obtain
		%$ -(\dot{p_1}(t), \dot{p_2}(t)) \in \textrm{co } \left(\partial_{(x,z)}(q_1(t) \cdot \bar{u}(t)) + \partial_{(x,z)}(q_2(t) L(\bar{x}(t),\bar{u}(t))) \right) $ 
		\[
		-(\dot{p}_1(t), \dot{p_2}(t)) \in q_2(t) \text{ co } \partial_{x} L(t,\bar{x}(t), \bar{u}(t)) \times \lbrace 0 \rbrace.
		\]
		Therefore, $\dot{p_2}(t)=0$ and $ - \dot{p_1}(t) \in q_2(t) \textrm{co } \partial_{x} L(t,\bar{x}(t), \bar{u}(t)).$
		We deduce that $ p_2(t) = q_2(t)=-1  $ and $\dot{p_1}(t) \in {\rm co }\ \partial_{x} L(t,\bar{x}(t), \bar{u}(t))$ a.e. $t$.
		\ \\
		\noindent
		The maximization condition \ref{item: weirsterass condition proof main thm chap 5} is equivalent to:
		\[
		q_1(t) \cdot \bar{u}(t) - L(t,\bar{x}(t), \bar{u}(t)) =  \mathop {\max }\limits_{u \in \mathbb{R}^n} \lbrace q_1(t) \cdot u - L(t,\bar{x}(t), u) \rbrace.
		\]
		Deriving the expression above in terms of $u$, and evaluating it at $u=\dot{\bar x}$, we obtain
		\[
		0 \in \textrm{co } \partial_{\dot{x}} \lbrace q_1(t) \cdot \bar{u}(t) - L(t,\bar{x}(t), \bar{u}(t)) \rbrace.
		\]
		Making use of the lower semicontinuity property of $v \mapsto L(t,\bar x(t),v)$ (cf. (\ref{item: CV2})) and using the sum rule \cite[page 45]{vinter_optimal_2010}, we have
		\[
		0 \in \text{ co }\lbrace \partial_{\dot{x}}(q_1(t) \cdot \dot{\bar{x}}(t)) +  \partial_{\dot{x}} (-L(t,\bar{x}(t), \bar{u}(t))) \rbrace \quad {\rm a.e.}
		\]
		Moreover, since $ \text{co }\partial(-f) = - \text{co }\partial(f)$, we deduce that
		\[ 0 \in q_1(t) - \textrm{co } \partial_{\dot{x}} (L(t,\bar{x}(t),\bar{u}(t)).\] Equivalently, \[ q_1(t) \in \textrm{co } \partial_{\dot{x}} L(t,\bar{x}(t), \bar{u}(t))   \quad \text{a.e. } t.  \] Condition (\ref{1_chap5}) of Theorem \ref{Thm1} is therefore satisfied. Condition (\ref{0_chap5}) is straightforward owing to \ref{item: element in the hybrid subdifferential proof main thm chap 5}. By consequence, the proof of Theorem \ref{Thm1} is complete. \qed

		\begin{remark}  \label{remark2}
			One may expect that the proof of Theorem \ref{Thm1} follows directly from results like \cite[Theorem 3.2]{fontes_normality_2015} for the particular case of free right-hand point constraint (eventually after expressing (\ref{problem: calculus of variations}) as an optimal control problem and by subsequently considering the augmented problem to cover the case of $W^{1,1}-$local minimizers). However, \cite[Theorem 3.2]{fontes_normality_2015} cannot be applied in our paper. Indeed, by inspecting \cite{fontes_normality_2015}, we notice that in order to apply the normality result \cite[Theorem 3.2]{fontes_normality_2015}, some regularity assumptions on the data should be satisfied, coupled with constraint qualifications (obviously weaker than the ones proposed in our paper (\ref{item: CQ1 normality ocp chap 5})-(\ref{item: CQ3 normality ocp chap 5}), and those discussed in the series of papers \cite{fontes_normal_2013}, \cite{lopes_constraint_2011}, \cite{ferreira_nondegenerate_1999}). In particular, \cite{fontes_normality_2015} assumes the following boundedness strong condition on the dynamics (referred as (H3) in \cite{fontes_normality_2015}): for a given $\delta>0$ and a local minimizer $\bar x(.)$,
			\[ \text{there exists } C_u \ge 0 \text{ such that } |f(t, x, u)| \le C_u \text{ for } x \in \bar x(t) + \delta \B , \ u \in U (t),
			\text{ and } t \in [S,T] \ .   \]
			It is clear that this regularity assumption is stronger than the one considered in our paper (cf. inequalities (\ref{boundedness boundary point}) and (\ref{boundedness initial point}) above), and it cannot be satisfied for our problem (\ref{ocp by state augmentation chap 5}). Indeed, the dynamic $\tilde{f}$ in problem (\ref{ocp by state augmentation chap 5}) is the following
			
			\[\tilde{f}{(t,(x(t),z(t)), u(t))} :=
			\begin{pmatrix}
			u(t) \\ L(t,x(t),u(t))
			\end{pmatrix} \qquad \text{where } u(t) \in \R^n \ . \]
			The dynamic $\tilde{f}$ satisfies assumption (H3) of \cite{fontes_normality_2015} only when $u(.)$ and $x(.)$ are bounded (owing to the boundedness of $L$ on bounded sets by (\ref{item: CV1})). By contrast, nothing guarantees, under the assumptions considered in our paper, that $\tilde{f}$ satisfies (H3) of \cite{fontes_normality_2015} for any $x$ near $\bar x(.)$ and for any $u \in U(t)$ which eventually has to coincide with the whole space $\R^n$, to derive the correct necessary conditions (in the normal form) for the class of problems considered in Theorem \ref{Thm1}.
		\end{remark}
		
		\section{Proof of Theorem \ref{Thm2} (Global Minimizers)}
		In this subsection we give details of a shorter proof based on a simple technique using the neighboring feasible trajectory result with linear estimate (initially introduced in \cite{rampazzo_theorem_1999}), while regarding the calculus of variations problems as an optimal control problem with final cost. The result we aim to prove is valid for global minimizers and under the stronger constraint qualification
		\begin{enumerate}[label=($\widetilde{CQ}$), ref= $\widetilde{CQ}$]
			\item \[\text{int }\ T_A(z) \neq \emptyset \ , \quad \text{for all } z \in \partial A.\]
		\end{enumerate}
%			We may assume at the outset of the proof that $L$ satisfies the Lipschitz continuity w.r.t. $x$ with $\epsilon'=+\infty$ (i.e. the Lipschitz continuity imposed on $L(t,.,u)$ is global). This is true if we consider the truncation function $\text{tr}_{y,\epsilon'} : \R^n \to \R^n$ defined to be
%			\[ \text{tr}_{y,\epsilon'}(x) :=  \begin{cases}
%			x \quad & \text{if } |x-y| < \epsilon' \\ y + \epsilon' \frac{x-y}{|x-y|} \quad & \text{if } |x-y| \ge \epsilon',
%			\end{cases}   \]
%			and we replace $L$ by its local expression
%			\[ \tilde{L}(x,v) : = L(t,\text{tr}_{\bar x(t), \epsilon'}(x), v).   \]
			We consider $\bar x(.)$ to be a global minimizer for the calculus of variations problem (\ref{problem: calculus of variations}).
%			\begin{equation}\begin{cases}\label{problem: calculus of variations second technique}
%					\begin{aligned}
%						& {\text{minimize}}
%						&&  \int_S^{T} L (x(t),\dot{x}(t)) \ dt  \\
%						&&& \hspace{-1.9cm} \text{over arcs } x \in  W^{1,1}([S,T], \mathbb{R}^{n}) \text{ satisfying} \\
%						&&& x(S) = x_0 \ , \\
%						%&&& \dot{x}(t) \in \dot{\bar{x}}(t) + M\mathbb{B} \quad & \text{for a.e. } t \in [S,T] \ , \\
%						&&& x(t) \in  A \quad  {\rm for \, all }\; t \in [S,T] \  . \end{aligned}\end{cases} \tag{$\widetilde{CV}$}\leqnomode
%			\end{equation}
%			\noindent
			We now employ the same standard argument (known as {\it state augmentation}) previously stated in the first proof technique. This allows to write the problem of calculus of variations (\ref{problem: calculus of variations}) as an optimal control problem. Indeed, by adding an extra absolutely continuous state variable \[ z(t)= \int_{S}^{t} L(s,x(s), \dot{x}(s))ds\]
			and by considering the dynamics $\dot{x} =u$, the problem (\ref{problem: calculus of variations}) can be written as the optimal control problem (\ref{problem: ocp second technique}):

			\begin{equation}\begin{cases}
					\begin{aligned} \label{problem: ocp second technique}
						& {\text{minimize}}
						&& z(T)  \\
						&&& \hspace{-1.9cm} \text{over } W^{1,1} \text{ arcs } (x(.),z(.)) \text{ satisfying}  \\
						&&& (\dot{x}(t), \dot z(t)) \in F(t,x(t)) \quad & {\rm a.e}\ t \in [S,T] \\
						&&& (x(S), z(S)) = (x_{0},0) \\
						&&&  (x(t), z(t)) \in A \times \mathbb{R} \quad & {\rm for \; all}\ t \in [S,T]  \  . \end{aligned}\end{cases}\tag{$\widetilde{P}$}\leqnomode
			\end{equation}
			where
			\begin{equation} \label{velocity set chap 5 second technique}        F(t,x) \ := \  \{ (u,L(t,x,u)) \ : \ u \in M \mathbb{B} \} \end{equation}
			for any $M>0$ large enough (for instance $M \ge 2 \| {\bar u(.)} \|_{L^\infty}$).
			
			\noindent
			It is easy to prove that if $\bar x(.)$ is a global minimizer for (\ref{problem: calculus of variations}), then $(\bar{x}(.), \bar{z}(.))$ is a global minimizer for (\ref{problem: ocp second technique}) where $\bar  z(t)= \int_{S}^{t} L(s,\bar x(s), {\dot{\bar x}}(s))ds$.
			
			\vskip2ex
			\noindent
			The proof of Theorem \ref{Thm2} is given in three steps. In Step 1, we show that the neighboring feasible trajectory theorem in \cite[Theorem 2.3]{bettiol_l$infty$_2012} holds true for our velocity set $F(.,.)$ and our constraint qualification $A\times\R$. In Step 2, we combine a penalization method with the $L^\infty-$linear estimate provided by \cite[Theorem 2.3]{bettiol_l$infty$_2012} which will permit to derive a minimizer for a problem involving a penalty term (in terms of the state constraint) in the cost. Step 3 is devoted to apply a standard Maximum Principle to an auxiliary problem and deduce the assertions of Theorem \ref{Thm2} in their normal form.
			\vskip2ex
			\noindent
			{\bf Step 1. } In this step, we prove that $(F(.,.),A\times \R)$ satisfies the hypothesis of \cite[Theorem 2.3]{bettiol_l$infty$_2012}. Indeed, the set of velocities is nonempty, of closed values owing to the Lipschitz continuity of $u \to L(t,x,u)$ (see \cite[Proposition 2.2.6]{clarke_optimization_1990}), and $F(.,x)$ is Lebesgue measurable for all $x\in \mathbb{R}^n$. Moreover, the set $F(t,x)$ is bounded on bounded sets, owing to the boundedness of $L(.,.,.)$ on bounded set, and to the boundedness of $\bar u(t)$ a.e. $t\in[S,T]$ (we recall that the minimizer $\bar x(.)$ is assumed to be Lipschitz). The Lipschitz continuity of $F(t,.)$ and the absolute continuity from the left of $F(.,x)$ is a direct consequence of assumption \ref{item: CV1'}. Finally, we observe that
			\[ F(t,x) \cap \text{int }T_{A\times \mathbb{R}}(x,z) \ = \ 
			 \\ \bigg( M \mathbb{B} \cap \text{int }T_A(x)   \bigg)  \times \mathbb{R} . \]
			The constraint qualification (\ref{CQ for calculus of variations_second_technique}) that we suggest
			\[\text{int }\ T_A(x) \neq \emptyset \quad \text{for all } x \in \partial A \ , \]	
			and the boundedness of $\|\bar u \|_{L^\infty} = \| \dot{\bar x} \|_{L^\infty}$ guarantee that for $M>0$ chosen large enough
				\[ F(t,x) \cap \text{int }T_{A\times \mathbb{R}}(x,z) \ \neq \ \emptyset . \]
			Therefore, \cite[Theorem 2.3]{bettiol_l$infty$_2012} is applicable.
			\vskip4ex
			\noindent
			{\bf Step 2.} In this step, we combine a penalization technique with the linear estimate given by \cite[Theorem 2.3]{bettiol_l$infty$_2012}. We obtain that $(\bar x, \bar z)$ is also a global minimizer for a new optimal control problem, in which an extra penalty term intervenes in the cost:
			\begin{lemma}\label{lemma: minimizer for a problem with penalty term in the cost, second technique chap 5} Assume that all hypotheses of Theorem \ref{Thm2}  are satisfied. Then, $(\bar x, \bar z)$ is a global minimizer for the problem:
				
				\begin{equation}\begin{cases}\label{problem: penalized ocp chap 5}\begin{aligned}
							& {\text{minimize}}
							&&  z(T) +  K \max\limits_{t\in[S,T]} d_{A \times \R}(x(t),z(t)) =: J(x(.),z(.))  \\
							&&& \hspace{-1.9cm} \text{over arcs }\ (x,z) \in  W^{1,1}([S,T], \mathbb{R}^{n}\times \R) \text{ satisfying} \\
							&&& (\dot x(t), \dot z(t)) \in F(t,x(t)) \quad  \text{a.e. }\ t\in[S,T] \\
							&&& (x(S),z(S)) = (x_0,0) , \quad x(S) \in A \ . \end{aligned}\end{cases} \tag{$\widetilde{P1}$}\leqnomode
				\end{equation} Here, $ K$ is the constant provided by \cite[Theorem 2.3]{bettiol_l$infty$_2012}.
			\end{lemma}
			
			\begin{proof}
				Suppose that there exists a global minimizer $(\hat{x}(.), \hat{z}(.))$ for (\ref{problem: penalized ocp chap 5}) such that
				\[  J(\hat{x}(.), \hat{z}(.)) < J(\bar{x}(.), \bar{z}(.))  . \] 
				Denote by $\hat \varepsilon := \max\limits_{t\in[S,T]} d_{A\times \R}(\hat{x}(t),\hat{z}(t)),$ the extent to which the reference trajectory $(\hat{x}(.),\hat{z}(.))$ violates the state constraint $A\times \R$. By the neighboring feasible trajectory result (with $L^\infty-$estimates) \cite[Theorem 2.3]{bettiol_l$infty$_2012}, there exists an $F-$trajectory $(x(.),z(.))$ and $ K>0$ such that
				\[
				\begin{cases}
				(x(S),z(S))=(x_0,0) \\
				(x(t),z(t)) \in A \times \R \quad \text{for all } t \in [S,T] \\ \| (x(.),z(.)) - (\hat x(.), \hat z(.)) \|_{L^\infty(S,T)} \le  K\hat \varepsilon.
				\end{cases}\]
				In particular
				\[ |z(T) - \hat{z}(T)| \le  K \hat{\varepsilon}.   \] Therefore,
				\begin{align*}
					z(T) \le \hat{z}(T) +  K \hat{\varepsilon} =  J(\hat x(.),\hat{z}(.)) < J(\bar x(.),\bar{z}(.)) = \bar z(T).
				\end{align*}
				But this contradicts the minimality of $(\bar x,\bar z)$ for (\ref{problem: ocp second technique}). The proof is therefore complete.
			\end{proof}
			A consequence of Lemma \ref{lemma: minimizer for a problem with penalty term in the cost, second technique chap 5} is the following:
			\begin{lemma}
				$\bar X(.) := (\bar{x}(.),\bar{z}(.)=\int_{S}^{t}L(s,\bar{x}(s),\dot{\bar x}(s))\ ds ,\bar w(.) \equiv 0)$ is a global minimizer for
				
				\begin{equation*}\begin{cases}\begin{aligned}
							& {\text{minimize}}
							&&  {g}(X(T)) \\
							&&& \hspace{-1.9cm} \text{over }\ X(.)=(x(.),z(.),w(.)) \in  W^{1,1}([S,T], \mathbb{R}^{n+2}) \text{ satisfying} \\
							&&& \dot{X}(t) = (\dot x(t), \dot z(t), \dot{w}(t)) \in G(t,X(t)) \quad  \text{a.e. } t\in[S,T] \\
							&&& {h}(X(t)) \le 0 \quad \text{for all } t\in [S,T] \\
							&&& (x(S),z(S),w(S)) \in  \{x_0\}   \times \{0\} \times \R^+ , \quad x(S) \in A
							\ . \end{aligned}\end{cases}
				\end{equation*}
				The cost function ${g}$ is defined by
				\[  {g}(X(T)):= z(T) + {K} w(T) . \]
				The multivalued function is defined by
				\[ G(t,X(t)) :=  \{ (u,L(t,x,u),0) \ : \ u \in M \B    \}   \] 
				and the function ${h}: \R^n \times \R \times \R \to \R$ (which provides the state constraint in terms of a functional inequality) is given by:
				\[{h}(X)= {h}(x,z,w) := d_{A}(x) - w \]
			\end{lemma} 
			\begin{proof}
				By contradiction, suppose that there exists a state trajectory $X(.):=(x(.),z(.),w(.))$ satisfying the state and dynamic constraints of the problem, such that
				\[   {g}(X(T)) <  {g}(\bar X(T)). \] Observe that $w(.)\equiv w \ge 0$, and the state constraint condition is equivalent to
				\[ \max\limits_{t\in[S,T]} d_{A} (x(t)) \le w .	\]
				Then, we would obtain
				\begin{align*}
					J(x(.),z(.)) & = z(T) + {K} \max\limits_{t\in[S,T]}d_{A}(x(t)) \\ & \le  z(T) +  K w \\ & = g(X(T)) < g(\bar{X}(T)) = J(\bar{x}(.),\bar z(.)).  
				\end{align*}	
				This contradicts the fact that $(\bar x(.),\bar z(.))$ is a global minimizer for (\ref{problem: penalized ocp chap 5}).
			\end{proof}

			\vskip4ex
			\noindent
			\textbf{Step 3.} In this step we apply known necessary optimality conditions: there exist costate arcs $P(.)=(p_1(.),p_2(.),p_3(.)) \in W^{1,1}([S,T],\mathbb{R}^{n+2})$ associated with the minimizer $(\bar x(.), \bar{z}(.), \bar{w}\equiv 0)$, a Lagrange multiplier $\lambda \ge 0$, a Borel measure $\mu(.): [S,T] \to \R$ and a $\mu-$integrable function $\gamma(.)=(\gamma_1(.),\gamma_2(.),\gamma_3(.))$ such that:
			\begin{enumerate}[label=(\roman{*}), ref= \roman{*}]
				\item $(\lambda, P(.),\mu(.)) \ne (0,0,0)$ ,
				\item $-\dot{P}(t) \in \text{co }\partial_{(x,z,w)} \bigg( Q(t) \cdot (\dot{\bar x}(t), L(t,\bar{x}(t), \dot{\bar x}(t)), 0) \bigg)$ ,
				\item $(P(S),-Q(T)) \in  \lambda \partial g (\bar x(T), \bar z(T), \bar w \equiv 0) + (\R  \times \R \times \R^-) \times (\{0\}\times \{0\}\times \{0\})$ ,
				\item $Q(t) \cdot (\dot{\bar x}(t), L(t,\bar{x}(t), \dot{\bar x}(t)), 0) =  \max\limits_{u\in M\B}  Q(t) \cdot (u, L(t,\bar x(t),u), 0) $
				\item \label{condition_v} ${\gamma}(t) \in \partial^>_{(x,z,w)} {h}(\bar{x}(t),\bar{z}(t), \bar{w}\equiv0)$ \; $\mu-$a.e. \quad and \quad supp$(\mu) \subset \{ t \ : \ {h}(\bar{x}(t),\bar{z}(t), \bar{w}\equiv0) =0   \}$.
			\end{enumerate}
			Here, $Q(.): [S,T] \to \R^{n+2}$ is the function
			\[ Q(t) := P(t) + \int_{[S,t]}\gamma(s) \ d\mu(s) \quad \text{for } t\in (S,T] \ .   \]
			Note that by condition (\ref{condition_v}), $\gamma(t) = (\gamma_1(t),\gamma_2(t),\gamma_3(t)) \in \partial^>_x d_A(\bar x(t)) \times \{0\} \times \{-1\} $ \; $\mu-$a.e.\\
			\noindent
			From the conditions above, we derive the following:
			\begin{enumerate}[label=(\roman{*})$'$, ref= (\roman{*})$'$]
				\item \label{item: nco1 second technique}$(\lambda, p_1(.),p_2(.), p_3(.),\mu(.)) \ne (0,0,0,0,0)$ ,
				\item $-\dot{p}_1(t) \in p_2(t)\text{co }\partial_{x} L(t,\bar{x}(t), \dot{\bar x}(t))$ , a.e. $t\in[S,T]$ \quad and \quad $\dot{p_2}=\dot{p_3} \equiv 0$ ,
				\item \label{item: nco3 second technique} $-(p_1(T) + \int_{S}^{T}{\gamma}_1(s) \ d\mu(s)) = 0$, \quad $p_3(S)=p_3 \le 0$, \quad $-(p_3(T)-\int_{[S,T]} d\mu(s)) = \lambda K$ \quad and \quad $p_2(T)=p_2=-\lambda$ ,
				\item \label{item: nco4 second technique} $(p_1(t) + \int_{S}^{t}{\gamma}_1(s) \ d\mu(s)) \cdot \dot{\bar x}(t) - \lambda L(t,\bar{x}(t),\dot{\bar{x}}(t))  = \max\limits_{u\in M\B} \big\{  (p_1(t) + \int_{S}^{t}{\gamma}_1(s) \ d\mu(s)) \cdot u - \lambda L(t,\bar{x}(t),u) \big\} $
				\item \label{item: nco5 second technique}${\gamma}_1(t) \in \partial^>_{x}d_{A}(\bar{x}(t))$ \; $\mu-$a.e. \quad and \quad supp$(\mu) \subset \{ t \ : \ {h}(\bar{x}(t),\bar{z}(t), \bar{w}\equiv0) =0   \}$ 
			\end{enumerate}
			We prove that conditions \ref{item: nco1 second technique}-\ref{item: nco5 second technique} apply in the normal form (i.e. with $\lambda=1$). Indeed, suppose that $\lambda=0$, then
			\[  -p_3 + \int_{[S,T]} d\mu(s) =0 \quad p_1=0 \quad p_2=0 \quad p_3=0.   \]
			But this contradicts the nontriviality condition \ref{item: nco1 second technique}. Therefore, the relations \ref{item: nco1 second technique}-\ref{item: nco5 second technique} apply in the normal form.
			Moreover, notice that the convexity of $L(t,x,.)$ yields that the maximality condition \ref{item: nco4 second technique} is verified globally (i.e. for all $u\in \R^n$), and by deriving it w.r.t. $u=\dot x$, and making use of the max rule (\cite[Theorem 5.5.2]{vinter_optimal_2010}) we deduce that
			\[   p_1(t) + \int_{[S,t]}\gamma_1(s) \ d\mu(s) \in \text{co }\partial_{\dot x}L(t,\bar{x}(t), \dot{\bar{x}}(t)) \quad \text{a.e. } t\in [S,T].   \]
			This permits to conclude the necessary optimality conditions in the normal form of Theorem \ref{Thm2}. 		\qed

\vskip2ex{\bf Acknowledgments:} The authors are thankful to the reviewer of the first version of the paper who suggested to use the $L^\infty-$linear distance estimate approach. Moreover, helpful comments and suggestions by Piernicola Bettiol are gratefully acknowledged.

\end{document}